\documentclass[12pt]{amsart}

\usepackage{amsthm}

\usepackage{amssymb}

\usepackage[english]{babel}
\usepackage{amsmath}
\usepackage[utf8]{inputenc}
\usepackage{eufrak}
\usepackage{amscd}
\usepackage{mathrsfs}
\usepackage{enumerate}
\usepackage{yhmath}

\usepackage{hyperref}

\newtheorem{theorem}{Theorem}[section]
\newtheorem{corollary}[theorem]{Corollary}
\newtheorem{lemma}[theorem]{Lemma}
\newtheorem{proposition}[theorem]{Proposition}

\theoremstyle{definition}
\newtheorem{definition}[theorem]{Definition}
\newtheorem{example}[theorem]{Example}

\theoremstyle{remark}
\newtheorem{remark}[theorem]{Remark}

\numberwithin{equation}{section}

\DeclareMathOperator*{\Sp}{\mathrm{Sp}}

\begin{document}

\title{The structure of algebraic Baer $^*$-algebras}

\author{Zsolt Sz\H{u}cs}
\address{Department of Differential Equations, Budapest University of Technology and Economics, M\H{u}egyetem rakpart 3., 1111 Budapest, Hungary}
\email{szucszs@math.bme.hu}

\author{Bal\'{a}zs Tak\'{a}cs}
\address{Department of Analysis, Budapest University of Technology and Economics, M\H{u}egyetem rakpart 3., 1111 Budapest, Hungary}
\email{takacsb@math.bme.hu}

%\author{Bal\'{a}zs Tak\'{a}cs}
%\address{Department of Analysis, Budapest University of Technology and Economics, M\H{u}egyetem rakpart 3., 1111 Budapest, Hungary}
%\email{takacsb@math.bme.hu}


\date{August 14, 2021.}



%    General info

%\date{July 17, 2021.}


\begin{abstract}
The purpose of this note is to describe when a general complex algebraic $^*$-algebra is pre-$C^*$-normed, and to investigate their structure when the $^*$-algebras are Baer $^*$-rings in addition to algebraicity. 
As a main result we prove the following theorem for complex algebraic Baer $^*$-algebras: every $^*$-algebra of this kind can be decomposed as a direct sum $M\oplus B$, where $M$ is a finite dimensional Baer $^*$-algebra and $B$ is a commutative algebraic Baer $^*$-algebra. The summand $M$ is $^*$-isomorphic to a finite direct sum of full complex matrix algebras of size at least $2\times2$. The commutative summand $B$ is $^*$-isomorphic to the linear span of the characteristic functions of the clopen sets in a Stonean topological space.  

As an application we show that a group $G$ is finite exactly when the complex group algebra $\mathbb{C}[G]$ is an algebraic Baer $^*$-algebra.
\end{abstract}

\maketitle

\subsubsection*{Published in Pacific Journal of Mathematics, Vol. 316 (2022), No. 2, 431–452.}

DOI: \url{https://doi.org/10.2140/pjm.2022.316.431}

\section{Introduction and preliminaries}\label{intro}

The authors investigated complex algebraic $^*$-algebras in the recently published article \cite{SzT}.  
The present paper is a continuation of this research, and our goal is that to derive new and significant properties of this class of $^*$-algebras (Theorems \ref{expreC}, \ref{main} and \ref{cg}). 

First we recall the notions in question. All rings and algebras are associative and not necessarily unital.
An element $a$ of an algebra $A$ over a field $\mathbb{F}$ is said to be \emph{algebraic}, if $a$ generates a finite dimensional subalgebra, or equivalently, $a$ is a root of a polynomial with coefficients from $\mathbb{F}$. 
The algebra $A$ is \emph{algebraic}, if every element of $A$ is algebraic. 
It is obvious that every subalgebra of an algebraic algebra is algebraic as well.
For general properties of algebraic algebras we refer the reader to \cite{J}.

An \emph{involution} of a ring $A$ is a mapping $^*:A\to A$ such that
\[
(a^*)^*=a;\ (a+b)^*=a^*+b^*;\ (ab)^*=b^*a^*\ \ \ \ \ (a,b\in A).
\] 
A \emph{$^*$-ring} is a ring with involution. 
We say that $A$ is a \emph{$^*$-algebra}, if it is an algebra over the complex numbers $\mathbb{C}$, and there is an involution $^*:A\to A$ with the additional property
\[
(\lambda a)^*=\overline{\lambda}a^*\ \ \ \ \ (\lambda\in\mathbb{C},\ a\in A).
\] 
So our convention in the paper is that a $^*$-algebra is always complex and not necessarily unital.
The set $\{a\in A\ |\ a^*=a\}$ of \emph{self-adjoint elements} is denoted by $A_{sa}$, while $A_+$ is the set of \emph{positive elements} of $A$, i.e., $a\in A_+$ iff $a=\sum_{j\in J}a_j^*a_j $ for some finite system $(a_j)_{j\in J}$ in $A$.

For a $^*$-algebra $A$ and an $a\in A$, the notation $\Sp_A(a)$ stands for the \emph{spectrum} of $a$. That is, 
\[
{\Sp}_A(a):=\{\lambda\in\mathbb{C}|\lambda\mathbf{1}-a\ \textrm{has no inverse in}\ A\}
\]
when $A$ has unit $\mathbf{1}$, and if $A$ is non-unital, then $\Sp_A(a)$ denotes the spectrum of $(0,a)$ in the standard unitization  $A^1:=\mathbb{C}\oplus A$ of $A$.
The $^*$-algebra $A$ is said to be \emph{hermitian}, if every selfadjoint element in $A$ has real spectrum.
$A$ is \emph{symmetric} (resp. \emph{completely symmetric}), if $\Sp_A(a^*a)\subseteq\mathbb{R}_+$ for any $a\in A$ (resp. $\Sp_A(x)\subseteq\mathbb{R}_+$ for any $x\in A_+$), see 9.8. in \cite{Palmer2}. 
Note here that a unital $^*$-algebra $A$ is symmetric iff for all $a\in A$ the element $\mathbf{1}+a^*a$ has inverse in $A$. 
For other general properties of $^*$-algebras and $C^*$-algebras see \cite{Dav}, \cite{Palmer1} and \cite{Palmer2}.

The presence of algebraic $^*$-algebras is natural in the theories of operator algebras and harmonic analysis.
Namely, the $^*$-algebra of finite rank operators on a complex Hilbert space (\cite{SzT}, Example 4.2) and the convolution $^*$-algebra of finitely supported complex functions defined on a discrete locally finite group (\cite{SzT}, Theorem 6.5) are algebraic $^*$-algebras. 
Moreover, algebraic $^*$-algebras appear in the important theory of approximately finite dimensional $C^*$-algebras. 
According to one of the definitions (\cite{Dav}, III.2), a $C^*$-algebra $\mathscr{A}$ is approximately finite dimensional if it has an increasing sequence $(A_n)_{n\in\mathbb{N}}$ of finite dimensional $^*$-subalgebras such that $A:=\cup_{n\in\mathbb{N}}A_n$ is norm-dense in $\mathscr{A}$. 
It is obvious that $A$ is an algebraic $^*$-algebra.

The examples above explain that it might be of interest to study
general properties of these $^*$-algebras.
The authors' recent paper \cite{SzT} contains many results in this context (Theorems 3.8 and 4.1, Proposition 3.10).
For instance, in section 3 we introduced the concept of an \emph{$E^*$-algebra}, that is, a $^*$-algebra $A$ with the following extension property: every representable positive functional defined on an arbitrary $^*$-subalgebra has a representable positive linear extension to $A$. 
Among other results, we showed that a $^*$-algebra is an $E^*$-algebra if and only if it is hermitian and every selfadjoint element is algebraic (\cite{SzT}, Theorem 3.8). In fact, more was proved: the class of $E^*$-algebras coincides with the class of completely symmetric algebraic $^*$-algebras (\cite{SzT}, Corollary 4.5). 

To continue the investigation in this direction, the recent paper focuses on two conditions which an algebraic $^*$-algebra $A$ may fulfil: 
\begin{enumerate}
\item a pre-$C^*$-norm exists on $A$;
\item $A$ is a Baer $^*$-algebra.
\end{enumerate}
In the subsections \ref{algpreintro} and \ref{Baerintro}, we make a short survey on these properties.
Our main result related to $(1)$ is Theorem \ref{expreC}, which is a characterization theorem. 
In relation to $(2)$ we prove Theorem \ref{main}, which is a structure theorem for algebraic Baer $^*$-algebras.
As an application of the latter theorem we show that for a group $G$ the group algebra $\mathbb{C}[G]$ is an algebraic Baer $^*$-algebra iff $G$ is finite (Theorem \ref{cg}).

\subsection{Algebraic pre-$C^*$-algebras}\label{algpreintro}

Let $A$ be a $^*$-algebra.
A submultiplicative seminorm $\sigma:A\to\mathbb{R}_+$ is a \emph{$C^*$-seminorm} (\cite{Palmer2}, 9.5), if 
\begin{equation}\label{ccsillagprop}
\sigma(a^*a)=\sigma(a)^2\ \ \ \ \ (a\in A).
\end{equation}
The kernel of $\sigma$, i.e., the $^*$-ideal $\{a\in A|\sigma(a)=0\}$ is denoted by $\ker\sigma$. 
If $\ker\sigma=\{0\}$, that is, $\sigma$ is a (possibly incomplete) norm, then we say that $\sigma$ is a \emph{pre-$C^*$-norm} and $(A,\sigma)$ is a \emph{pre-$C^*$-algebra}.

One of the most important concepts in the general theory of $^*$-algebras is the following (\cite{Palmer2}, Definition 10.1.1).
For every $a\in A$, let
\begin{equation}\label{greatest}
\gamma_A(a):=\sup\{\Vert\pi(a)\Vert|\pi\ \textrm{is a $^*$-representation of $A$ on a Hilbert space}\}.
\end{equation}
If $\gamma_A$ is finite valued, then it is called the \emph{Gelfand-Naimark seminorm of $A$}, and it is the greatest $C^*$-seminorm on $A$ (\cite{Palmer2}, Theorem 10.1.3).

The concrete examples of algebraic $^*$-algebras above have a common property: they all admit a pre-$C^*$-norm. 
Of course, in general, an algebraic $^*$-algebra does not have a pre-$C^*$-norm. 
For example, Theorem 9.7.22 in \cite{Palmer2} produces many finite dimensional (hence algebraic) $^*$-algebras which not possess a pre-$C^*$-norm.
So a natural question arises: if $A$ is a given algebraic $^*$-algebra, then what kind of assumptions guarantee the existence of a pre-$C^*$-norm on $A$? 
In section \ref{chars}, we give purely algebraic answers to this question (Theorem \ref{expreC}), and in fact, each of these conditions actually characterize the existence of a pre-$C^*$-norm on $A$.
We recall them here.

\begin{definition}\label{properdef}
The involution of a $^*$-ring $A$ is \emph{proper} (\cite{Palmer2}, page 990; \cite{Be1}, §2), if for any $a\in A$ the equation $a^*a=0$ implies that $a=0$.
\end{definition}

By \eqref{ccsillagprop}, it is immediate that every pre-$C^*$-algebra has proper involution.
In Theorem \ref{expreC}, we show for an algebraic $^*$-algebra that the properness of the involution guarantees that the $^*$-algebra admits a pre-$C^*$-norm.

The following concept was first studied by J. von Neumann (see the Introduction in \cite{G}); for general properties see \cite{Be2} and \cite{G}.

\begin{definition}\label{vNregular}
A ring $A$ is said to be \emph{von Neumann regular}, if for every $a\in A$ there is an $x\in A$ such that $axa=a$.
\end{definition}

Von Neumann regularity and hermicity imply for an algebraic $^*$-algebra that the $^*$-algebra has a pre-$C^*$-norm (Theorem \ref{expreC}). 
In subsection \ref{noAbelian}, which deals with Baer $^*$-algebras without Abelian summand, von Neumann regularity is very important.

Before the next definition we recall the notion of \emph{annihilators} (\cite{Be1}, page 12): if $A$ is a ring, $S\subseteq A$ is an arbitrary subset, then the \emph{right annihilator} of $S$ is
\[
\mathrm{ann}_{rA}(S):=\{y\in A|ay=0\ \forall a\in S\},
\] 
which is obviously a right ideal of $A$.
The \emph{left annihilator} $\mathrm{ann}_{lA}(S)$ of $S$ can be defined similarly (which is a left ideal); a non-trivial ideal of $A$ is an \emph{annihilator ideal}, if it is a left annihilator of some set $S\subseteq A$.
(We drop $A$ from the index, when it is clear from the context in which algebra the annihilator is considered.)

\begin{definition}\label{Rickartdef}
A $^*$-ring $A$ is said to be a
\begin{itemize}
\item[(a)] \emph{Rickart $^*$-ring} (\cite{Be1}, §3), if for any $a\in A$ the right annihilator of $\{a\}$  is a principal right ideal generated by a projection $f$, that is, 
\[
\mathrm{ann}_r(\{a\})=fA:=\{fy|y\in A\}
\]
with an $f\in A$ such that $f=f^*=f^2$.
\item[(b)] \emph{weakly Rickart $^*$-ring} (\cite{Be1}, §5), if for every $a\in A$ there is a projection $e\in A$ with the following properties:
\begin{enumerate}
\item[(1)] $ae=a$;
\item[(2)] $ay=0$ implies $ey=0$ for every $y\in A$.
\end{enumerate}
If $A$ is a weakly Rickart $^*$-ring, the conditions $(1)$ and $(2)$ uniquely determines the projection $e$, and it is called \emph{the right projection} of $a$; in notation: $e=\mathrm{RP}(a)$.
The \emph{left projection} $\mathrm{LP}(a)$ can be obtained similarly.

A (weakly) Rickart $^*$-algebra is a $^*$-algebra which is also a (weakly) Rickart $^*$-ring.
\end{itemize}
\end{definition}
By §3 and §5 in \cite{Be1}, we have the following properties.
A $^*$-ring $A$ is a Rickart $^*$-ring if and only if $A$ is a weakly Rickart $^*$-ring with unity; in this case the definitions above give $\mathrm{RP}(a)=\mathbf{1}-f$. 
A very important observation is that the involution of a weakly Rickart $^*$-ring is proper, hence algebraic $^*$-algebras which are also weakly Rickart $^*$-algebras admit pre-$C^*$-norms (Theorem \ref{expreC}). 

For a $^*$-ring $A$ let $\mathbf{P}(A)$ be the set of projections in $A$; let $\leq$ denote the natural ordering, i.e., $p\leq q$ iff $p=pq=qp$ ($p,q\in \mathbf{P}(A)$).
Note here that there is an ordering on the set of self-adjoint elements (\cite{Be1}, Definition 8 on page 70 and §50).
In the results of this paper these orderings coincide on $\mathbf{P}(A)$, hence we use the same notation: for $x,y\in A_{sa}$ $x\leq y$ means that $y-x\in A_+$.

The projections  of a weakly Rickart $^*$-algebra form a lattice with the following operations (\cite{Be1}, Proposition 7 on page 29):
\[
e\vee f=f+\mathrm{RP}(e-ef);\ e\wedge f=e-\mathrm{LP}(e-ef)\ \ \ (e,f\in\mathrm{P}(A)).
\]

\subsection{Algebraic Baer $^*$-algebras}\label{Baerintro} 

As we mentioned in the preceding subsection, we prove that algebraic weakly Rickart $^*$-algebras have pre-$C^*$-norms (Theorem \ref{expreC}). 
In Section \ref{struct}, we investigate algebraic $^*$-algebras with the following stronger assumption, which was introduced by I. Kaplansky to give an algebraic axiomatization of von Neumann algebras. 
(See the introductions of \cite{Be1} and \cite{G}.)  

\begin{definition}\label{Baer}
Let $A$ be a $^*$-ring.
We say that $A$ is a \emph{Baer $^*$-ring} (\cite{Be1}, §4), if for every non-empty subset $S\subseteq A$ the right annihilator $\mathrm{ann}_r(S)$ of $S$ is a principal right ideal generated by a projection. 
A \emph{Baer $^*$-algebra} is a $^*$-algebra which is also a Baer $^*$-ring.
\end{definition}

For Baer $^*$-rings the reader is referred to the fundamental works of S. K. Berberian \cite{Be1} and \cite{Be2}.

Our main result in the context of algebraic Baer $^*$-algebras is Theorem \ref{main}, which gives a complete description for this class.
To obtain this theorem we collect some indispensable properties of Baer $^*$-rings here (see \cite{Be1}).

By the definitions, it is obvious that a Baer $^*$-ring $A$ is a Rickart $^*$-ring, hence the involution of $A$ is proper and $A$ is unital.  
A very important property of the projection lattice of a Baer $^*$-ring is contained in the following statement (\cite{Be1}, Proposition 1 in §4):

\begin{proposition}\label{RickartBaer}
The following conditions on a $^*$-ring $A$ are equivalent.
\begin{itemize}
\item[(i)]$A$ is a Baer $^*$-ring;
\item[(ii)]$A$ is a Rickart $^*$-ring whose projections form a complete lattice.
\end{itemize}
\end{proposition}

From a given Baer $^*$-ring or $^*$-algebra we can easily obtain "new" Baer $^*$-rings (\cite{Be1}, §4).
If $A$ is a Baer $^*$-ring (resp. $^*$-algebra), then $eAe$ is a Baer $^*$-subring (resp. $^*$-subalgebra) for every projection $e\in\mathbf{P}(A)$, as well as the \emph{bicommutant} $S'':=\{a\in A|as=sa\ \forall s\in S\}$ of every $^*$-closed subset $S$ of $A$.  
 It is obvious that the image of a Baer $^*$-ring via a $^*$-isomorphism (i.e., injective $^*$-homomorphism) is also a Baer $^*$-ring.

The more concrete examples of Baer $^*$-rings are the so-called $AW^*$-algebras, that is, $C^*$-algebras which are Baer $^*$-algebras (see \cite{Be1} and \cite{SWb}).
For instance, every von Neumann algebra is an $AW^*$-algebra (\cite{Be1}, Proposition 9 on page 24), but
for us the following two examples are the most important.

\begin{example}\label{finitedim}
The class of finite dimensional Baer $^*$-algebras is precisely the class of finite dimensional $C^*$-algebras.
Indeed, a Baer $^*$-algebra $A$ has proper involution, hence finite dimensionality implies that $A$ is a $C^*$-algebra (\cite{Palmer2}, Theorem 9.7.22). Moreover, $A$ can be written as a finite direct sum of standard complex matrix $^*$-algebras (see also \cite{Dav}, page 74): 
\begin{equation}\label{matrix}
A\cong\oplus_{k=1}^mM_{n_k}(\mathbb{C}).
\end{equation}
Furthermore, it is clear that full matrix $^*$-algebras over $\mathbb{C}$ are Baer $^*$-algebras, as well as their finite direct sums.
\end{example}

\begin{example}\label{kommutativ}
By Theorem 1 in §7 of \cite{Be1} and the well-known Gelfand-Naimark theorem, a commutative $C^*$-algebra $\mathscr{B}$ is an $AW^*$-algebra if and only if it is $^*$-isomorphic to the $C^*$-algebra $\mathscr{C}(T;\mathbb{C})$ of the complex valued continuous functions on a \emph{Stonean topological space} $T$ (that is, $T$ is compact Hausdorff and \emph{extremally disconnected}: every open set has open closure).

For a Stonean topological space $T$, the $C^*$-algebra $\mathscr{C}(T;\mathbb{C})$ naturally contains a norm-dense algebraic Baer $^*$-subalgebra.
Indeed, the projections of $\mathscr{C}(T;\mathbb{C})$ are exactly the characteristic functions of the clopen (closed and open) sets in $T$ (\cite{Be1}, page 40 and 41), and these form a complete lattice.
Hence, if $\widehat{B}$ denotes the linear span of the characteristic functions of the clopen sets, then it is a dense (\cite{Be1}, Proposition 1 in §8) unital $^*$-subalgebra with complete projection lattice. 
From the commutativity, it is easy to see that $\widehat{B}$ is algebraic, moreover the supremum-norm is a pre-$C^*$-norm on $\widehat{B}$.
Our result (Theorem \ref{expreC}) implies that $\widehat{B}$ is a Rickart $^*$-algebra, thus $\widehat{B}$ is a Baer $^*$-algebra by Proposition \ref{RickartBaer}.
\end{example}

Now one may ask the following question: are there algebraic Baer $^*$-algebras that have different structure than the ones in the previous examples?   
The answer is in 
our result, Theorem \ref{main}. 
It states that every algebraic Baer $^*$-algebra is a direct sum of a finite dimensional Baer $^*$-algebra and a commutative algebraic Baer $^*$-algebra as in Example \ref{kommutativ}.

\section{Existence of pre-$C^*$-norms on algebraic $^*$-algebras}\label{chars}

Our recently published paper \cite{SzT} contains an equivalent condition for a pre-$C^*$-algebra to be algebraic (Theorem 4.1).
As it was mentioned in \ref{algpreintro} of the Introduction, we assume algebraicity in this section and we study conditions which imply that an algebraic $^*$-algebra has a pre-$C^*$-norm (Theorem \ref{expreC}). 
This is our main result in this part.

The following statement about algebraic algebras is elementary (see Lemma 3.2 in \cite{SzT} and Lemma 3.12 in \cite{Passman2}).

\begin{lemma}\label{spegyen}
Let $A$ be a complex algebraic algebra. 
For an element $b\in A$, denote by $A_b$ the subalgebra generated by $\mathbf{1}$ and $b$ (resp. $b$) if $A$ is unital (resp. $A$ is non-unital).  

Then for every $b\in A$, the spectrum $\Sp_A(b)$ is non-empty, finite and the following is true:
\begin{equation}\label{egyenlosp}
{\Sp}_{A}(b)=\left\{\begin{array}{ll}
\Sp_{A_{b}}(b), & \text{when $A$ is unital}, \\   
\Sp_{A_{b}}(b)\cup\{0\}, & \text{when $A$ is non-unital}.
\end{array}\right.
\end{equation}
Furthermore, for any $a,b\in A$
\begin{align}\label{szorzatsp}
{\Sp}_A(ab)={\Sp}_A(ba)
\end{align}
holds.
\end{lemma}

If $A$ is a complex algebra, then $A_J$ stands for the Jacobson radical of $A$. 
Following Palmer's terminology in \cite{Palmer1} and \cite{Palmer2}, $A$ is semisimple iff $A_J=\{0\}$.

The next statement collects several results from the authors' paper \cite{SzT}.

\begin{theorem}\label{Studia}
Let $A$ be a $^*$-algebra such that every selfadjoint element of $A$ is algebraic.
Then the following are equivalent.
\begin{itemize}
\item[(i)] $A$ is hermitian.
\item[(ii)] $A_J=\ker\gamma_A$.
\end{itemize}
If these properties hold, then $A$ is a completely symmetric algebraic $^*$-algebra such that the elements of the Jacobson radical $A_J$ are nilpotent.
Furthermore, the Gelfand-Naimark seminorm $\gamma_A$ in \eqref{greatest} is finite valued, hence it is the greatest $C^*$-seminorm on $A$.  
\end{theorem}
\begin{proof}
See \cite{SzT}: the equivalence comes from Theorem 3.8 (ii$a$) and (iic); the properties follow from Proposition 3.10. $(b)$, $(g)$, $(f)$ and Theorem 3.8 (iii).
\end{proof}

Our main result on general algebraic $^*$-algebras is the following statement.
We use it often in the Baer $^*$-algebra case, Section \ref{struct}.
Note here that the properties listed below (with the exception of (i)) are purely algebraic conditions.

\begin{theorem}\label{expreC}
Let $A$ be an algebraic $^*$-algebra.
The following assertions are equivalent.
\begin{itemize}
\item[(i)] There exists a pre-$C^*$-norm on $A$.
\item[(ii)] The involution of $A$ is proper.
\item[(iii)] $A$ is hermitian and semisimple.
\item[(iv)]  Every non-zero selfadjoint $b\in A_{sa}$ can be written in a finite sum
\[
b=\sum_{j\in J}\lambda_je_j,
\] 
where $(\lambda_j)_{j\in J}$ is a finite system of non-zero distinct real numbers and $(e_j)_{j\in J}$ is a finite system of non-zero orthogonal projections in $A$.
\item[(v)] $A$ is a weakly Rickart $^*$-ring.

\item[(vi)] $A$ is hermitian and von Neumann regular. 
\end{itemize}
\end{theorem}
\begin{proof}
We prove the equivalence of (i), (ii) and (iii) first.

(i)$\Rightarrow$(ii): 
By the $C^*$-property \eqref{ccsillagprop}, the involution is proper on any $^*$-algebra which admits a pre-$C^*$-norm. 

(ii)$\Rightarrow$(iii):
We show that $A$ is hermitian if the involution of $A$ is proper. 
Let us fix a selfadjoint $b\in A_{sa}$; we have to prove that $\Sp_{A}(b)\subseteq\mathbb{R}$.

The equation in \eqref{egyenlosp} implies that it is enough to see the inclusion $\Sp_{A_{b}}(b)\subseteq\mathbb{R}$.
Since $b$ is algebraic, the subalgebra $A_b$ generated by $b$ (which is trivially a $^*$-subalgebra)  is finite dimensional, and has proper involution from the assumption. So Theorem 9.7.22 in \cite{Palmer2} forces $A_b$ to be hermitian, hence $\Sp_{A_{b}}(b)\subseteq\mathbb{R}$ holds.

Now we conclude the semisimplicity. 
The Jacobson radical $A_J$ is a $^*$-subalgebra of $A$, hence $A_J=(A_J)_{sa}+\mathbf{i}(A_J)_{sa}$ is true ($\mathbf{i}$ is the complex imaginary unit).
By properness, $(A_J)_{sa}=\{0\}$. 
As we recalled in Theorem \ref{Studia} above, every element of $A_J$ is nilpotent. 
Thus for an $a\in(A_J)_{sa}$, let $n\in\mathbb{N}$ be the minimal positive integer such that $a^n=0$. 
If $n\geq 2$, then $2n-2\geq n$ and $n-1\geq1$. 
So $a^{2n-2}=0$, and since $a^*=a$, it follows that $(a^{n-1})^*(a^{n-1})=0$. 
The involution is proper, which implies $a^{n-1}=0$. 
This contradicts the minimality of $n$, thus $n=1$, i.e., $a=0$.   
 
(iii)$\Rightarrow$(i): This follows from
Theorem \ref{Studia}, since $\gamma_A$ is a pre-$C^*$-norm on $A$.

Now we show the implications (i)$\Rightarrow$(iv)$\Rightarrow$(iii).

(i)$\Rightarrow$(iv): Let $\Vert\cdot\Vert$ be a pre-$C^*$-norm on $A$. 
Fix a non-zero $b\in A_{sa}$ selfadjoint element, and let $B$ be the $^*$-subalgebra generated by $b$, which is finite dimensional since $b$ is algebraic. 
Hence $(B,\Vert\cdot\Vert|_{B})$ is a non-zero finite dimensional commutative $C^*$-algebra.
Since $b$ is selfadjoint, the well-known commutative Gelfand-Naimark theorem and spectral theory imply the existence of a finite system $(\lambda_j)_{j\in J}$ of non-zero distinct real numbers and a finite system $(e_j)_{j\in J}$ of non-zero orthogonal projections in $B$ such that
\[
b=\sum_{j\in J}\lambda_je_j.
\] 

(iv)$\Rightarrow$(iii): 
Let $b\in A_{sa}$ be an arbitrary selfadjoint element. 
From the assumption (iv) write $b=\sum_{j\in J}\lambda_je_j$ with a finite system $(\lambda_j)_{j\in J}$ of non-zero distinct real numbers and a finite system $(e_j)_{j\in J}$ of non-zero orthogonal projections in $A$. 
Then an easy calculation shows that $b$ is a root of the polynomial 
\[
p(z)=z\prod_{j\in J}(z-\lambda_j).
\]
Hence the spectral mapping property implies that 
\[
{\Sp}_A(b)\subseteq\{\lambda_j|j\in J\}\cup\{0\}\subseteq\mathbb{R},
\] 
which means that $A$ is hermitian. 

To show semisimplicity, assume that $b\in(A_J)_{sa}\setminus\{0\}$.
The Jacobson radical is an ideal, hence for any $j_0\in J$ from the orthogonality we have
\[
e_{j_0}=\left(\frac{1}{\lambda_{j_0}}e_{j_0}\right)b\in A_J.
\]
But this is absurd, since the radical does not contain non-zero idempotent elements. 
From this it follows that $(A_J)_{sa}=\{0\}$, thus $A_J=\{0\}$. 
The equivalence of (i), (ii), (iii) and (iv) has been proved; let us proceed to the properties (v) and (vi).

(ii) \& (iv)$\Rightarrow$(v): 
According to Definition \ref{Rickartdef}, we have to prove for any $a\in A$ that there is a projection $e\in A$ with the following properties:
\begin{enumerate}
\item[(1)] $ae=a$;
\item[(2)] $ay=0$ implies $ey=0$ for every $y\in A$.
\end{enumerate}
We may suppose that $a\neq0$. 
From the assumption (ii) we infer that $a^*a\neq0$, thus by (iv) we may write
\[
a^*a=\sum_{j\in J}\lambda_je_j
\] 
with a finite system $(\lambda_j)_{j\in J}$ of non-zero distinct real numbers and a finite system $(e_j)_{j\in J}$ of non-zero orthogonal projections in $A$. 
Define the projection $e$ by the following equation:
\begin{equation}\label{RP}
e:=\sum_{j\in J}e_j.
\end{equation}
By orthogonality, it is indeed a projection.

Let us prove property (1), that is, $ae=a$.
The involution is proper, hence 
\[
ae=a\ \Leftrightarrow \ ae-a=0 \ \Leftrightarrow \ (ae-a)^*(ae-a)=0 \Leftrightarrow  
\]
\[
ea^*ae-a^*ae-ea^*a+a^*a=0, 
\]
and the last equation is true, since the obvious equalities 
\[
ea^*ae=a^*ae=ea^*a=a^*a 
\]
hold.

$(2)$: Let $y\in A$ be arbitrary with $ay=0$. 
Then $a^*ay=0$, that is,
\[
\left(\sum_{j\in J}\lambda_je_j\right)y=0.
\]
Since every $\lambda_j\neq0$, by orthogonality we conclude that
\[
0=\left(\sum_{j\in J}\frac{1}{\lambda_j}e_j\right)\left(\sum_{j\in J}\lambda_je_j\right)y=\left(\sum_{j\in J}e_j\right)y=ey,
\]
hence $A$ is a weakly Rickart $^*$-ring.

(v)$\Rightarrow$(ii):
Any weakly Rickart $^*$-ring has a proper involution (subsection \ref{algpreintro} in the Introduction).

(ii), (iii) $\&$ (iv)$\Rightarrow$(vi): 
Hermicity follows from (iii). 
To see the von Neumann regularity, let $a\in A$ be an arbitrary non-zero element. 
We seek an $x\in A$ such that $axa=a$. 
Similar to the preceding proofs, by (ii) and (iv) we may write 
\[
a^*a=\sum_{j\in J}\lambda_je_j
\] 
with a finite system $(\lambda_j)_{j\in J}$ of non-zero distinct real numbers and a finite system $(e_j)_{j\in J}$ of non-zero orthogonal projections in $A$. 
Define the element $x$ by the equation
\[
x:=\left(\sum_{j\in J}\frac{1}{\lambda_j}e_j\right)a^*.
\]
Then
\[
axa=a\left[\left(\sum_{j\in J}\frac{1}{\lambda_j}e_j\right)a^*\right]a=a\left(\sum_{j\in J}\frac{1}{\lambda_j}e_j\right)\left(\sum_{j\in J}\lambda_je_j\right)=
\]
\[
a\left(\sum_{j\in J}e_j\right)=a.
\]
(The last equation is property $(1)$ in the proof of the implication (ii) $\&$ (iv)$\Rightarrow$(v).)

(vi)$\Rightarrow$(iii): A von Neumann regular algebra is always semisimple: if $a\in A_J$, $x\in A$ and $axa=a$, then $axax=ax$, that is, $ax$ is an idempotent in the radical. Thus, $ax=0$, which implies $a=0$.
\end{proof}

\begin{remark}
We note here that if there is a complete pre-$C^*$-norm on the algebraic $^*$-algebra $A$, then $A$ is finite dimensional.
In fact, by Kaplansky's Lemma 7 in \cite{Kap}, an infinite dimensional semisimple Banach algebra has an element with infinite spectrum. Thus, by Lemma \ref{spegyen}, every semisimple algebraic Banach algebra is finite dimensional.
Furthermore, a von Neumann regular Banach algebra is finite dimensional as well (this is also Kaplansky's theorem: \cite{Palmer1}, Theorem 2.1.18).
\end{remark}

\begin{remark}\label{tulok}
Let $A$ be an algebraic $^*$-algebra.
\begin{enumerate}
\item Semisimplicity does not imply that $A$ is hermitian (see Theorem 9.7.22 in \cite{Palmer2}), and the same is true for von Neumann regularity (semisimple finite dimensional $^*$-algebras are von Neumann regular, including the non-hermitian ones).

\item If $A$ satisfies the properties of the previous theorem (hence it is a weakly Rickart $^*$-ring), then for an element $a\in A$ the right projection $\mathrm{RP}(a)$ of $a$ is the projection $e$ in \eqref{RP}. 
That is, if 
\[
a^*a=\sum_{j\in J}\lambda_je_j
\] 
for a finite system $(\lambda_j)_{j\in J}$ of non-zero distinct real numbers and a finite system $(e_j)_{j\in J}$ of non-zero orthogonal projections in $A$, then  
\begin{equation*}\label{RP2}
\mathrm{RP}(a)=\sum_{j\in J}e_j=\mathrm{RP}(a^*a).
\end{equation*}
Furthermore, from the proof of the theorem it can be easily seen that the projections $e_j$ $(j\in J)$ are in the subalgebra generated by $a^*a$. 
It follows that the right projection of any $a\in A$ is in the $^*$-subalgebra generated by $a$.  

\item By means of Theorem \ref{Studia}, one can show that property (iv) actually characterizes algebraic pre-$C^*$-algebras in the class of $^*$-algebras.
Namely, a $^*$-algebra $A$ is an algebraic $^*$-algebra with a pre-$C^*$-norm if and only if every non-zero selfadjoint $b\in A_{sa}$ can be written in a finite sum
\begin{equation}\label{projszum}
b=\sum_{j\in J}\lambda_je_j,
\end{equation}
where $(\lambda_j)_{j\in J}$ is a finite system of non-zero distinct complex numbers and $(e_j)_{j\in J}$ is a finite system of non-zero orthogonal projections in $A$. 
In this case every normal element $b$ of $A$ has the form \eqref{projszum}. 
Moreover,
\[
{\Sp}_A(b)\cup\{0\}=\{\lambda_j|j\in J\}\cup\{0\}
\]   
holds for the spectrum of $b$.
The numbers $\lambda_j$ are real when $b$ is selfadjoint, and they are non-negative if $b$ is a positive element, since
$A$ is completely symmetric by Theorem \ref{Studia}.
\end{enumerate}
\end{remark}

Now we list further properties of general algebraic pre-$C^*$-algebras (some of them were shown in the authors' paper \cite{SzT}).
These are indispensable in Section \ref{struct}.

\begin{proposition}\label{preCprops}
If $A$ is an algebraic pre-$C^*$-algebra, then the following assertions hold.
\begin{itemize}
\item[(1)] The pre-$C^*$-norm on $A$ is unique.
%\item $A$ is a semiprime algebra (\cite{Palmer1}, Definition 4.4.1), that is, $A$ has no non-zero nilpotent ideal.
\item[(2)] Denote by $C^*(A)$ the completion of $A$ with respect to the unique pre-$C^*$-norm.
Then for every closed ideal $I\subseteq C^*(A)$ the equation $I=\overline{A\cap I}$ holds, where the bar means the closure in the norm of $C^*(A)$.

\item[(3)] $A$ satisfies the \textbf{(EP)-axiom} (\cite{Be1}, Definition 1 on page 43), that is, for any $a\in A\setminus\{0\}$ there exists a $w\in\{a^*a\}''$ such that $w=w^*$ and $(a^*a)w^2$ is a non-zero projection.

\item[(4)] $A$ satisfies the \textbf{(UPSR)-axiom} (\cite{Be1}, Definition 10 on page 70), that is, for any $x\in A_+$ there is a unique $y\in A_+$ such that $y^2=x$ and in addition $y\in\{x\}''$.
%\item symmetric (sőt, completely symmetric); mi a spektrum a projektorosból!

\item[(5)] $A$ is the linear span of its projections.
\end{itemize}
Moreover, if $A$ has a unit element, then:
\begin{itemize}
\item[(6)] $a^*a\leq\Vert a\Vert^2\mathbf{1}$ holds for any $a\in A$. 

\item[(7)] $A$ is a \textbf{directly finite} algebra, that is, for any $a,b\in A$ the equation $ab=\mathbf{1}$ implies that $ba=\mathbf{1}$. 
\end{itemize}
\end{proposition}
\begin{proof}
$(1)$: Theorem 4.1 $(a)$ in \cite{SzT} shows this.

$(2)$: This statement is Theorem 4.1 $(f)$ in \cite{SzT}.

$(3)$: Let $a\in A$ be a non-zero element. 
Since $a^*a\geq0$, by Theorem \ref{expreC} and Remark \ref{tulok} $(3)$, we may write 
\[
a^*a=\sum_{j\in J}\lambda_je_j
\] 
with a finite system $(\lambda_j)_{j\in J}$ of non-zero distinct positive numbers and a finite system $(e_j)_{j\in J}$ of non-zero orthogonal projections in $A$. 
Then the element 
\[
w=\sum_{j\in J}\frac{1}{\sqrt{\lambda_j}}e_j
\]
 is selfadjoint (in fact, positive), moreover orthogonality implies
\[
(a^*a)w^2=\left(\sum_{j\in J}\lambda_je_j\right)\left(\sum_{j\in J}\frac{1}{\lambda_j}e_j\right)=\sum_{j\in J}e_j=\mathrm{RP}(a),
\] 
which is a non-zero projection, as desired.
The property $w\in\{a^*a\}''$ for the bicommutant  also holds, since $w$ is an element of the $^*$-subalgebra generated by $a^*a$ according to  Remark \ref{tulok} $(2)$ .

$(4)$: Let $x\in A_+$ be a non-zero positive element. 
Then $x$ is of the form 
\[
x=\sum_{j\in J}\lambda_je_j
\] 
with a finite system $(\lambda_j)_{j\in J}$ of non-zero distinct positive numbers and a finite system $(e_j)_{j\in J}$ of non-zero orthogonal projections in $A$. 
Obviously the element 
\[
y:=\sum_{j\in J}\sqrt{\lambda_j}e_j
\]
 is positive and orthogonality shows that $y^2=x$.

To see the uniqueness, let $z\in A$ be a positive element such that $z^2=x$. 
Let $A_z$ (resp. $A_x$) denote the $^*$-subalgebra generated by $z$ (resp. $x$).
Then $A_x\subseteq A_z$, moreover these are finite dimensional $C^*$-algebras by the algebraicity. 
From Remark \ref{tulok} $(2)$, it follows that $y\in A_x$, hence $y\in A_z$.
But $y^2=x=z^2$ and it is well-known that every positive element has a \emph{unique} positive square root  in a $C^*$-algebra, thus $y=z$. 
The inclusion $y\in\{x\}''$ is also clear from $y\in A_x$.

$(5)$:  The equality $A=A_{sa}+\mathbf{i}A_{sa}$ holds for the $^*$-algebra $A$. 
Thus, Theorem \ref{expreC} (iv) shows that every element of $A$ can be written as a linear combination of the projections in $A$.

For the rest of the proof, assume that $A$ is unital.

$(6)$: For an arbitrary $a\in A$, the $^*$-subalgebra generated by $a^*a$ and $\mathbf{1}$ is a finite dimensional $C^*$-algebra.
Hence spectral theory and the $C^*$-property imply the inequality $a^*a\leq\Vert a\Vert^2\mathbf{1}$. 

$(7)$: This follows from \eqref{szorzatsp} in Lemma \ref{spegyen}.
\end{proof}

As we noted in the Introduction \ref{Baerintro}, every Baer $^*$-ring is a Rickart $^*$-ring. Hence 
for an algebraic $^*$-algebra which is also a Baer $^*$-ring, the properties in Theorem \ref{expreC} and Proposition \ref{preCprops} come true. 
We formulate this in the next statement.

\begin{corollary}\label{Baerprops}
If $A$ is an algebraic Baer $^*$-algebra, then $A$ is a directly finite, von Neumann regular and completely symmetric $^*$-algebra with a unique pre-$C^*$-norm $\Vert\cdot\Vert$, satisfying the (EP)- and (UPSR)-axioms. 
Moreover, $A$ is the linear span of its projections, and for any $a\in A$ the inequality $a^*a\leq\Vert a\Vert^2\mathbf{1}$ holds.
\end{corollary}

To close this section we present the following lemma which characterizes finite dimensionality of algebraic $^*$-algebras with pre-$C^*$-norms.
It is used in the proof of Theorem \ref{noabelian}. 

\begin{lemma}\label{Kaplansky}
An algebraic pre-$C^*$-algebra $A$ is finite dimensional if and only if every orthogonal system of non-zero projections in $A$ is finite.
\end{lemma}
\begin{proof}
If $A$ is finite dimensional, then it is obvious that every orthogonal system of non-zero projections in $A$ is finite, since such a system is linearly independent.

For the converse, suppose that every orthogonal system of non-zero projections in $A$ is finite. 
Note first that it is enough to prove finite dimensionality in the case of unital $A$.
Indeed, if $A$ is not unital, then let $A^1:=\mathbb{C}\oplus A$ be the standard unitization of the $^*$-algebra $A$. 
Now it is clear that $A^1$ is also an algebraic $^*$-algebra with a pre-$C^*$-norm (\cite{Palmer2}, Proposition 9.1.13 $(b)$).
Assume that $(\lambda_k,p_k)_{k\in\mathbb{N}}$ is an infinite orthogonal system of projections in $A^1$.
Then orthogonality shows for every $k,m\in\mathbb{N}$, $k\neq m$ that 
\[
(0,0)=(\lambda_k,p_k)(\lambda_m,p_m)=(\lambda_k\lambda_m,\lambda_kp_m+\lambda_mp_k+p_kp_m),
\]
hence $\lambda_k\neq0$ holds for at most one $k\in\mathbb{N}$.
Dropping out this possible exception, the system has the form $(0,p_k)_{k\in\mathbb{N}}$.
Now it is clear that $(p_k)_{k\in\mathbb{N}}$ would be an infinite orthogonal system of projections in $A$, which is impossible according to the assumption. 
Thus, all of the assertions we made for $A$ are true for $A^1$. 
Moreover if $A^1$ is finite dimensional, then $A$ is also finite dimensional.

Suppose that $A$ is a unital algebraic pre-$C^*$-algebra such that every orthogonal system of non-zero projections in $A$ is finite.
Then it is easy to see that $A$ does not contain a strictly increasing sequence of non-zero projections. 
Indeed, if $(q_k)_{k\in\mathbb{N}}$ is such sequence, then the projections $p_0:=q_0$, $p_{k+1}:=q_{k+1}-q_{k}$ $(k\in\mathbb{N})$ would be form an infinite orthogonal system of non-zero projections.

By Theorem \ref{expreC} (vi), the $^*$-algebra $A$ is von Neumann regular, hence every finitely generated right ideal $I$ of $A$ is of the form $eA$ with an idempotent $e\in A$ (see Lemma 1.3 on page 68 in \cite{Passman2} or Theorem 1.1 in \cite{G}). 
Furthermore, since the involution of $A$ is proper (Theorem \ref{expreC} (ii)), Proposition 3 on page 229 in \cite{Be1} implies that we may assume that $e$ is a projection. 
This and the preceding paragraph imply that $A$ satisfies the ascending chain condition for finitely generated right ideals.
Hence, by Lemma 2.9 on page 80 in \cite{Passman2}, we obtain that $A$ is Artinian. 
In fact, the proof of this lemma and the arguments above imply that every right ideal of $A$ is generated by a projection. 
Thus, the unital $A$ is a \emph{modular annihilator algebra} (\cite{Palmer1}, Definition 8.4.6), because it is semisimple (hence semiprime: \cite{Palmer1}, Definition 4.4.1) and obviously satisfies property $(a)$ of Theorem 8.4.5 in \cite{Palmer1}.
Now Proposition 8.4.14 of \cite{Palmer1} implies that $A$ is finite dimensional, since it is a normed algebra.
\end{proof}

\section{Algebraic Baer $^*$-algebras}\label{struct}

The main purpose of this section is to answer the following questions. Theorem \ref{main} answers them.
\begin{enumerate}
\item In the light of Corollary \ref{Baerprops}, what other properties an algebraic Baer $^*$-algebra has?  
\item Are there examples of algebraic Baer $^*$-algebras different than the ones given in Examples \ref{finitedim} and \ref{kommutativ}? 
\end{enumerate}

The lattice of projections of a Baer $^*$-ring is complete (\cite{Be1}, §15) and these rings have well-developed structure theory. We recall some relevant definitions.
%When we are dealing with Baer $^*$-rings, a powerful tool may come to our mind, namely, the structure theory thanks to the completeness of the projection lattice (\cite{Be1}, §15).
%This is be our way, so let us recall two definitions in this context.
We say that a $^*$-ring is \emph{Abelian} if all of its projections are central.
A $^*$-ring $A$ is \emph{properly non-Abelian}, if the only Abelian central projection is $0$, i.e, if for a central projection $e\in A$ the $^*$-ring $eAe$ is Abelian, then $e=0$. 

By Corollary \ref{Baerprops}, an algebraic Baer $^*$-algebra $A$ is Abelian iff $A$ is commutative, since it is the linear span of its projections (the same equivalence is true for $AW^*$-algebras: \cite{Be1}, Examples on page 90).

The key to our structure Theorem \ref{main} is the following decomposition (\cite{Be1}, Theorem 1 (2) in §15): 

\begin{theorem}\label{strukszum}
Let $A$ be a Baer $^*$-ring.
Then there exists a unique central projection $h\in A$ such that in the decomposition
\begin{equation*}
A=(\mathbf{1}-h)A+hA
\end{equation*}
the summand $M:=(\mathbf{1}-h)A$ is a properly non-Abelian Baer $^*$-subring, while $B:=hA$ is an Abelian Baer $^*$-subring.

If $B=\{0\}$, then we say that $A$ has no Abelian summand.
\end{theorem}

By this theorem and the observations above, an algebraic Baer $^*$-algebra $A$ can be decomposed as a sum of algebraic Baer $^*$-subalgebras 
\begin{equation}\label{AbelnonAbel}
A=M\oplus B,
\end{equation}
where $M$ has no Abelian summand and $B$ is commutative.
Thus, if we fully analyze the properly non-Abelian and the commutative case separately, we get a complete description of general algebraic Baer $^*$-algebras.
This has been done in the following subsections.

\subsection{Algebraic Baer $^*$-algebras without Abelian summands}\label{noAbelian}

The following statement is a reformulation of Corollary 3 on page 231 in \cite{Be1} (the assumptions for $A$ can be found in the statement of this corollary and at the beginning of §51).
Note here that a unital $^*$-ring $A$ is finite iff $x^*x=\mathbf{1}$ implies $xx^*=\mathbf{1}$ for any $x\in A$.

\begin{lemma}\label{231}
Let $A$ be a finite, von Neumann regular and symmetric Baer $^*$-ring without Abelian summand, satisfying the (EP)- and (UPSR)-axioms. 
Assume that $A$ contains a central element $i\in A$ with the properties $i^2=-\mathbf{1}$ and $i^*=-i$.

If, for any $a\in A$, there is a positive integer $k$ such that $a^*a\leq k\mathbf{1}$, then every system of non-zero orthogonal projections in $A$ is finite. 
\end{lemma}

We are ready to prove our result on algebraic Baer $^*$-algebras that have no Abelian summand.

\begin{theorem}\label{noabelian}
Let $M$ be a $^*$-algebra.
The following are equivalent. 
\begin{itemize}
\item[(i)]$M$ is an algebraic Baer $^*$-algebra without Abelian summand. 
\item[(ii)]$M$ is a finite dimensional Baer $^*$-algebra without Abelian summand.
\end{itemize}
If one (hence all) of the properties holds for $M$, then $M$ is a finite direct sum of full complex matrix algebras of size at least $2\times2$.
\end{theorem}
\begin{proof}
(i)$\Rightarrow$(ii):
By Corollary \ref{Baerprops}, $M$ satisfies all of the assertions of the previous lemma (with $i:=\mathbf{i}\mathbf{1}$), hence every system of non-zero orthogonal projections in $M$ is finite. 
Now Lemma \ref{Kaplansky} implies the finite dimensionality of $M$. 

(ii)$\Rightarrow$(i): Finite dimensional algebras are algebraic.

As we mentioned in Example \ref{finitedim}, finite dimensional Baer $^*$-algebras are finite direct sum of full matrix $^*$-algebras over $\mathbb{C}$.
Since $M$ has no Abelian summand, then the size of every matrix algebra in the decomposition \eqref{matrix} of Example \ref{finitedim} must be at least $2\times2$. 
\end{proof}

\subsection{Commutative algebraic Baer $^*$-algebras}\label{kommutat}

Our first result shows that the completion with respect to the unique pre-$C^*$-norm (Corollary \ref{Baerprops}) is also a Baer $^*$-algebra in the commutative case.

\begin{lemma}\label{closure}
Let $B$ be a commutative Baer $^*$-algebra which is algebraic.
If $C^*(B)$ stands for the completion of $B$ with respect to its unique pre-$C^*$-norm, then $C^*(B)$ is a commutative $AW^*$-algebra, and every projection of $C^*(B)$ is in $B$.
\end{lemma}
\begin{proof}
We prove that $C^*(B)$ is a Baer $^*$-algebra by Definition \ref{Baer}.
Since the algebras in question are commutative, the left and right attributes for the annihilators are redundant, so we have to prove for a fixed non-empty set $S\subseteq C^*(B)$ that 
\[
\mathrm{ann}_{C^*(B)}(S)=\{a\in C^*(B)|Sa=\{0\}\}=fC^*(B)
\]  
for some projection $f\in C^*(B)$. 
We show that such a projection exists in $B$.

For the non-empty set $B\cap\mathrm{ann}_{C^*(B)}(S)$ the Baer property implies that there is projection $e\in B$ such that
\[
eB=\mathrm{ann}_B\left(B\cap\mathrm{ann}_{C^*(B)}(S)\right).
\]
We state that for the projection $f:=\mathbf{1}-e\in B$ the equation
\begin{equation*}
\mathrm{ann}_{C^*(B)}(S)=C^*(B)f
\end{equation*}
is true.
To see this, it is enough to claim the following equality:
\begin{equation}\label{kommBa}
eC^*(B)=\mathrm{ann}_{C^*(B)}\left(\mathrm{ann}_{C^*(B)}(S)\right).
\end{equation}
Indeed, if \eqref{kommBa} holds, then taking annihilators on both sides gives by Proposition 1 (3) in §3 of \cite{Be1} that
\[
C^*(B)(\mathbf{1}-e)=\mathrm{ann}_{C^*(B)}(eC^*(B))=\mathrm{ann}_{C^*(B)}(S).
\]
Now we prove \eqref{kommBa}.
Let $a\in \mathrm{ann}_{C^*(B)}(S)$ be an arbitrary element.
Note first that every annihilator is obviously a closed ideal in the commutative $C^*$-algebra $C^*(B)$.
Thus, $B\cap\mathrm{ann}_{C^*(B)}(S)$ is dense in $\mathrm{ann}_{C^*(B)}(S)$ by Proposition \ref{preCprops} $(2)$, since $B$ is algebraic. 
So there is a sequence $(a_n)_{n\in\mathbb{N}}$ in $B\cap\mathrm{ann}_{C^*(B)}(S)$ such that $a_n\to a$. 
Since $e\in \mathrm{ann}_B\left(B\cap\mathrm{ann}_{C^*(B)}(S)\right)$ we infer that $ea_n=0$ for every $n\in\mathbb{N}$.
This implies $0=ea_n\to ea$, that is, 
\[
e\in \mathrm{ann}_{C^*(B)}\left(\mathrm{ann}_{C^*(B)}(S)\right), 
\]
thus the inclusion 
\[
eC^*(B)\subseteq\mathrm{ann}_{C^*(B)}\left(\mathrm{ann}_{C^*(B)}(S)\right)
\] 
follows.

For the reversed inclusion, let $x\in\mathrm{ann}_{C^*(B)}\left(\mathrm{ann}_{C^*(B)}(S)\right)$ be arbitrary. 
The latter set is a closed ideal in $C^*(B)$, so using Proposition \ref{preCprops} $(2)$ again we get a sequence $(x_n)_{n\in\mathbb{N}}$ in $B\cap \mathrm{ann}_{C^*(B)}\left(\mathrm{ann}_{C^*(B)}(S)\right)$ such that $x_n\to x$.
Thus, $x_n\in B$ and $x_n\mathrm{ann}_{C^*(B)}(S)=\{0\}$ for every $n\in\mathbb{N}$. 
In particular, $x_n(B\cap\mathrm{ann}_{C^*(B)}(S))=\{0\}$. 
But this exactly shows that $x_n\in \mathrm{ann}_B\left(B\cap\mathrm{ann}_{C^*(B)}(S)\right)=eB$ for every $n\in\mathbb{N}$.
Taking the limit of the sequence $(x_n)_{n\in\mathbb{N}}$, we have $x\in\overline{eB}=eC^*(B)$, so \eqref{kommBa} has been proved.

To see that every projection of $C^*(B)$ is in $B$, let $p\in\mathbf{P}(C^*(B))$ be arbitrary.
The preceding proof shows that 
\[
\mathrm{ann}_{C^*(B)}(\{p\})=qC^*(B)
\]
with a projection $q$ in $B$.
But in a unital commutative ring it is well known that the annihilator of an idempotent $p$ is the ideal generated by $\mathbf{1}-p$, hence $q=\mathbf{1}-p$ (\cite{Be1}, Proposition 1 in §1), so $p=\mathbf{1}-q\in B$. 
\end{proof}

Our next theorem characterizes algebraic Baer $^*$-algebras among commutative $^*$-algebras.

\begin{theorem}\label{abelian}
Let $B$ be a commutative $^*$-algebra.
The following statements are equivalent.
\begin{itemize}
\item[(i)]$B$ is an algebraic Baer $^*$-algebra.
\item[(ii)]There exists a Stonean topological space $T$ such that $B$ is $^*$-isomorphic to the linear span of the characteristic functions of the clopen sets in $T$. 
\end{itemize}
\end{theorem}
\begin{proof}
(i)$\Rightarrow$(ii): Suppose that $B$ is an algebraic Baer $^*$-algebra.
Then from the previous lemma, it follows that $C^*(B)$ is an $AW^*$-algebra. 
Hence, by Theorem 1 in §7 of \cite{Be1} and the Gelfand-Naimark theorem, $C^*(B)$ is $^*$-isomorphic to the $C^*$-algebra $\mathscr{C}(T;\mathbb{C})$ of the complex valued continuous functions on a Stonean topological space $T$. 
If $\widehat{B}$ denotes the image of $B$ via the Gelfand-transform, the lemma also concludes that all of the projections in $\mathscr{C}(T;\mathbb{C})$ are actually contained in $\widehat{B}$.
These projections are the characteristic functions of the clopen sets in $T$ (\cite{Be1}, page 40 and 41).
Since an algebraic Baer $^*$-algebra is the linear span of its projections (Corollary \ref{Baerprops}), we obtain (ii).

(ii)$\Rightarrow$(i):
The proof of this direction was discussed in Example \ref{kommutativ}.
If $\widehat{B}$ denotes the linear span of the characteristic functions of the clopen sets in a Stonean topological space $T$, then by the arguments in the example we get that $\widehat{B}$ is an algebraic Baer $^*$-algebra.
Hence, if $B$ is $^*$-isomorphic to $\widehat{B}$, then $B$ is an algebraic Baer $^*$-algebra as well.
\end{proof}

\subsection{General algebraic Baer $^*$-algebras}\label{general}

Now we are in position to prove our structure theorem for algebraic Baer $^*$-algebras.

\begin{theorem}\label{main}
 The following conditions are equivalent for a $^*$-algebra $A$.
\begin{itemize}
\item[(i)]$A$ is an algebraic Baer $^*$-algebra.
\item[(ii)]$A$ can decomposed as a sum $M\oplus B$, where $M$ is $^*$-isomorphic to a finite direct sum of full complex matrix algebras of size at least $2\times2$, while the summand $B$ is $^*$-isomorphic to the linear span of the characteristic functions of the clopen sets in a Stonean topological space $T$. 
\end{itemize}
\end{theorem}
\begin{proof}
(i)$\Rightarrow$(ii): 
According to \eqref{AbelnonAbel}, we can decompose $A$ into a sum 
\[
A=M\oplus B,
\]
where $M$ has no Abelian summand and $B$ is commutative.
Now from Theorems \ref{noabelian} and \ref{abelian}, the properties in (ii) immediately follow.

(ii)$\Rightarrow$(i): If $A=M\oplus B$ as in (ii), then $M$ and $B$ are algebraic Baer $^*$-algebras by Theorems \ref{noabelian} and \ref{abelian}. It is clear that their direct sum is an algebraic Baer $^*$-algebra as well.
\end{proof}

\begin{corollary}\label{AWcor}
The $C^*$-algebras which contain a dense algebraic Baer $^*$-algebra are exactly the $AW^*$-algebras with finite codimensional Abelian summand.  
\end{corollary}
\begin{proof}
Suppose that $\mathscr{A}$ is a $C^*$-algebra containing a dense algebraic Baer $^*$-algebra $A$.
By Corollary \ref{Baerprops}, there is a unique pre-$C^*$-norm on $A$. 
The density and the uniqueness imply that the completion $C^*(A)$ of $A$ with respect to this norm is actually $\mathscr{A}$. 
According to Theorem \ref{main}, we may write $A=M\oplus B$, where $M$ is a finite direct sum of full complex matrix algebras of size at least $2\times2$, while the summand $B$ is $^*$-isomorphic to the linear span of the characteristic functions of the clopen sets in a Stonean topological space $T$. 
Hence this form and the finite dimensionality of $M$ clearly imply  that
\[
C^*(A)=M\oplus C^*(B).
\]
We infer that $C^*(B)$ is an Abelian $AW^*$-algebra from Lemma \ref{closure}.
Since $M$ is a finite dimensional $AW^*$-algebra without commutative ideals, it follows that $C^*(A)$ is an $AW^*$-algebra with finite codimensional Abelian summand.   

Now let $\mathscr{A}$ be an $AW^*$-algebra with finite codimensional Abelian summand. By Theorem \ref{strukszum}, $\mathscr{A}$ can be decomposed as a sum
\[
\mathscr{A}=M\oplus\mathscr{B},
\]
where $M$ is a finite dimensional $AW^*$-algebra without Abelian summand and $\mathscr{B}$ is an Abelian (commutative) $AW^*$-algebra.
By Example \ref{kommutativ} and Theorem \ref{abelian}, $\mathscr{B}$ contains a dense algebraic Baer $^*$-algebra $B$, hence the algebraic Baer $^*$-algebra $A:=M\oplus B$ is dense in $\mathscr{A}$.   
\end{proof}

\section{An application to complex group algebras}\label{group}

Let $G$ be a group and let $\mathbb{F}$ be a field.
The group algebra $\mathbb{F}[G]$ consists of all formal finite sums of the form $\sum_{g\in G}\lambda_gg$, where $\lambda_g\in\mathbb{F}$. 
For $a,b\in\mathbb{F}[G]$, $a=\sum_{g\in G}\lambda_gg$, $b=\sum_{h\in G}\mu_hh$ and $\lambda\in\mathbb{F}$, the operations are defined by
\[
a+b:=\sum_{g\in G}(\lambda_g+\mu_g)g;\ ab:=\sum_{g,h\in G}(\lambda_g\mu_h)gh;\ \lambda a:=\sum_{g\in G}(\lambda\lambda_g)g.
\] 
If $\mathbb{F}=\mathbb{C}$, then an involution can be defined by
\[
a^*=\left(\sum_{g\in G}\lambda_gg\right)^*:=\sum_{g\in G}\overline{\lambda_g}g^{-1}.
\]
It is easy to see that $\mathbb{C}[G]$ is a $^*$-algebra with proper involution, moreover it admits a pre-$C^*$-norm.
(For the theory of group rings and locally compact groups we refer the reader to \cite{Passman2}, \cite{Palmer1} and \cite{Palmer2}.)

Our result specifies which groups have the property that their group algebras over $\mathbb{C}$ are algebraic Baer $^*$-algebras.

\begin{theorem}\label{cg}
If $G$ is a group, then the following statements are equivalent.
\begin{itemize}
\item[(i)]The complex group algebra $\mathbb{C}[G]$ is an algebraic Baer $^*$-algebra. 
\item[(ii)]$G$ is finite.
\end{itemize}
\end{theorem}
\begin{proof}
(i)$\Rightarrow$(ii): The elements of $G$ constitute an algebraic basis for $\mathbb{C}[G]$, hence we need to prove that $\mathbb{C}[G]$ is finite dimensional.
From Theorem \ref{main}, we obtain that $\mathbb{C}[G]$ splits into a direct sum $M\oplus B$, where $M$ is a finite dimensional Baer $^*$-algebra, $B$ is a commutative algebraic Baer $^*$-algebra which is $^*$-isomorphic to the linear span of the characteristic functions of the clopen sets in a Stonean topological space $T$. 
We examine two cases: 
\begin{itemize}
\item[(I)] $M=\{0\}$; 
\item[(II)] $M\neq\{0\}$. 
\end{itemize}
Case (I): if $M=\{0\}$, then $\mathbb{C}[G]=B$, i. e., $\mathbb{C}[G]$ is commutative. 
It is well known that this occurs exactly when $G$ is commutative.
Regarding $G$ as a discrete group, its dual group $\widehat{G}$ is a compact group (see \cite{Dav}, Chapter VII.1).
The Fourier-Gelfand transform (which is a $^*$-homomorphism) maps $\mathbb{C}[G]$ injectively onto a dense $^*$-subalgebra of $\mathscr{C}(\widehat{G};\mathbb{C})$.
Thus, the latter $C^*$-algebra contains a dense algebraic Baer $^*$-algebra.
Corollary \ref{AWcor} implies that $\mathscr{C}(\widehat{G};\mathbb{C})$ is an $AW^*$-algebra.
From this, it follows that $\widehat{G}$ is a Stonean topological space (Example \ref{kommutativ}), that is, compact and extremally disconnected. 
Now, from Theorem 1 in \cite{Rag}, we obtain that $\widehat{G}$ is discrete.
%Together with the compactness these conclude that $\widehat{G}$ is finite, hence $G$ is finite as well.
As $\widehat{G}$ is compact and discrete, $\widehat{G}$ is finite.

Case (II): if $M\neq\{0\}$, then $B$ is an annihilator ideal of $\mathbb{C}[G]$ (since $B$ is the left/right annihilator of $M$), which is finite codimensional.
By Passman's Theorem 3.1 in \cite{Passman}, $G$ is finite.

(ii)$\Rightarrow$(i): If $G$ is finite, then $\mathbb{C}[G]$ is a finite dimensional $^*$-algebra with a $C^*$-norm, thus it is an algebraic Baer $^*$-algebra.
\end{proof}

An analogous statement can be proved in the case of locally compact groups.

\begin{corollary}
Let $G$ be a locally compact group and let $\beta$ be a left Haar-measure of $G$.  
If $\mathscr{K}_{\beta}(G;\mathbb{C})$ stands for the convolution $^*$-algebra of compactly supported complex valued functions on $G$, then the following statements are equivalent. 
\begin{itemize}
\item[(i)] $\mathscr{K}_{\beta}(G;\mathbb{C})$ is an algebraic Baer $^*$-algebra.
\item[(ii)] $G$ is finite.
\end{itemize}
\end{corollary}
\begin{proof}
(i)$\Rightarrow$(ii): The conditions in (i) imply that $G$ is discrete. 
Indeed, Theorem 6.5 in \cite{SzT} 
or the presence of the unit element in the Baer $^*$-algebra $\mathscr{K}_{\beta}(G;\mathbb{C})$ ensures that $G$ is discrete. 
Now the group algebra $\mathbb{C}[G]$ can be identified with $\mathscr{K}_{\beta}(G;\mathbb{C})$, hence Theorem \ref{cg} applies. 

(ii)$\Rightarrow$(i): Obvious.
\end{proof}

\begin{remark}Let $G$ be a group with unit $\mathbf{1}_G$.
\begin{enumerate}
\item The conditions "algebraic" and "being a Baer $^*$-algebra" are independent for group algebras over $\mathbb{C}$.
For example, if $G$ is a finitely generated torsion-free commutative group, then for an arbitrary field $\mathbb{F}$ the group algebra $\mathbb{F}[G]$ has no zero-divisors (\cite{Passman2}, Lemma 1.1 in §1).
Thus, in the case of $\mathbb{F}=\mathbb{C}$, the group algebra $\mathbb{C}[G]$ is obviously a Baer $^*$-algebra. 
Moreover, for a $g\in G\setminus\{\mathbf{1}_G\}\subseteq\mathbb{C}[G]$, we have that $g$ has infinite order. So $\{g^n|n\in\mathbb{Z}\}$ is a linearly independent system in the group algebra, hence $g$ is not an algebraic element.

In fact, a theorem of Herstein (\cite{Passman2}, Theorem 3.11 in §3) states that if $\mathbb{F}$ is a field with zero characteristic, then $\mathbb{F}[G]$ is algebraic if and only if $G$ is locally finite.
Thus, for a non-finite locally finite group $G$, our Theorem \ref{cg} implies that $\mathbb{C}[G]$ is an algebraic pre-$C^*$-algebra which is not a Baer $^*$-algebra.  

\item In the absence of an involution, we say that a ring $A$ is \emph{Baer}, if the right annihilator $\mathrm{ann}_r(S)$ of every non-empty subset $S\subseteq A$ is a principal right ideal generated by an idempotent (\cite{Be2}).
By definition, a Baer $^*$-ring is a Baer ring, but the converse is not true in general (9.1.42 in \cite{Palmer2}).
%For example, let $A$ be the $^*$-algebra $\mathcal{S}(2)$ (\cite{Palmer2}, 9.1.42), that is, the complex algebra $\mathbb{C}\oplus\mathbb{C}$ with coordinate-wise operations and involution $(\lambda,\mu)^*:=(\overline{\mu},\overline{\lambda})$.
%It is easy to see that this $^*$-algebra is a Baer ring but not a Baer $^*$-ring (the only projections in $\mathcal{S}(2)$ are $(1,1)$ and $(0,0)$).

However, if a $^*$-ring is a Baer ring which is
von Neumann regular with proper involution (\cite{Be2}, Proposition 1.13) or symmetric (\cite{Be1},  Ex. 5A in page 25), then it is a Baer $^*$-ring.
Since complex group algebras $\mathbb{C}[G]$ have proper involution, the properties "Baer ring" and "Baer $^*$-ring" are equivalent for the algebraic complex group algebras (Theorems \ref{Studia} and \ref{expreC} imply that such algebras are symmetric and von Neumann regular). 
%Since complex group algebras $\mathbb{C}[G]$ have proper involution, for the algebraic ones the properties "Baer ring" and "Baer $^*$-ring" are equivalent (by Theorems \ref{Studia} and \ref{expreC} they are symmetric and von Neumann regular).
So, by the above mentioned theorem of Herstein and Theorem \ref{cg},
a locally finite group $G$ is finite if and only if $\mathbb{C}[G]$ is a Baer ring.
\end{enumerate}
\end{remark}

\section*{Acknowledgement}

The authors thank the anonymous referee for the insightful comments that improved the presentation the paper.

\bibliographystyle{amsplain}

\end{document}